\documentclass[12pt,a4paper]{article}
\usepackage[T1]{fontenc}
\usepackage{lmodern}
\usepackage[utf8]{inputenc}
\usepackage{amsthm}
\usepackage{amsmath,amssymb}
\usepackage{graphics}
\usepackage{lscape,graphicx}
\usepackage{enumitem}
\usepackage{amsfonts}
\usepackage[T1]{fontenc}
\usepackage{longtable}
\usepackage{authblk}
\usepackage{lipsum}
\usepackage{epsfig}
\usepackage{float}
\usepackage{hyperref}
\usepackage{comment}
\usepackage[normalem]{ulem}
\usepackage{multicol}
\usepackage{fancyhdr}
\usepackage[english]{babel}
\usepackage{mathrsfs}
\usepackage{tocloft}
\usepackage{xcolor}
\usepackage{titletoc}
\usepackage{mathtools}
\usepackage[toc,page]{appendix}

\addto{\captionsenglish}{}
\makeatletter
\def\thm@space@setup{%
  \thm@preskip=1cm plus .5cm minus .5cm
  \thm@postskip=.5cm plus .6cm minus .5cm 
}

\makeatother

\newtheorem{thm}{Theorem}
\newtheorem{lma}{Lemma}

\newtheorem{prop}{Proposition}
\newtheorem{rmk}{Remark}

\numberwithin{thm}{section}
\numberwithin{lma}{section}
\numberwithin{dfn}{section}
\numberwithin{cor}{section}
\numberwithin{rmk}{section}
\numberwithin{prop}{section}

\newcommand*{\thmref}[1]{Theorem~\ref{#1}}
\newcommand*{\lmaref}[1]{Lemma~\ref{#1}}

\newcommand*{\propref}[1]{Proposition~\ref{#1}}

\title{On the number of prime factors with a given multiplicity over $h$-free and $h$-full numbers}

\date{}

\begin{document}

\author{Sourabhashis Das, Wentang Kuo, Yu-Ru Liu}

\newcommand{\Addresses}{{
  \bigskip
  \footnotesize

  Sourabhashis Das (Corresponding author), Department of Pure Mathematics, University of Waterloo, 200 University Avenue West, Waterloo, Ontario, Canada, N2L 3G1. \\
  Email address: \texttt{s57das@uwaterloo.ca}

  \medskip

  Wentang Kuo, Department of Pure Mathematics, University of Waterloo, 200 University Avenue West, Waterloo, Ontario, Canada, N2L 3G1. \\
  Email address: \texttt{wtkuo@uwaterloo.ca}
  
  \medskip

  Yu-Ru Liu, Department of Pure Mathematics, University of Waterloo, 200 University Avenue West, Waterloo, Ontario, Canada, N2L 3G1. \\
  Email address: \texttt{yrliu@uwaterloo.ca}

}}

%
%
%

\maketitle 

\begin{abstract}
Let $k$ and $n$ be natural numbers. Let $\omega_k(n)$ denote the number of distinct prime factors of $n$ with multiplicity $k$ as studied by Elma and the third author \cite{el}. We obtain asymptotic estimates for the first and the second moments of $\omega_k(n)$ when restricted to the set of $h$-free and $h$-full numbers. We prove that $\omega_1(n)$ has normal order $\log \log n$ over $h$-free numbers, $\omega_h(n)$ has normal order $\log \log n$ over $h$-full numbers, and both of them satisfy the Erd\H{o}s-Kac Theorem. Finally, we prove that the functions $\omega_k(n)$ with $1 < k < h$ do not have normal order over $h$-free numbers and $\omega_k(n)$ with $k > h$ do not have normal order over $h$-full numbers.
\end{abstract}

\section{Introduction}

For\footnotetext{\textbf{2020 Mathematics Subject Classification: 11N05, 11N37.}}\footnotetext{\textbf{Keywords: prime factors, prime-counting functions, normal order, Erd\H{o}s-Kac theorem.}}\footnotetext{The research of W. Kuo and Y.-R. Liu are supported by NSERC discovery grants.} a natural number $n$, let the prime factorization of $n$ be given as
\begin{equation}\label{factorization}
n = p_1^{s_1} \cdots p_r^{s_r},
\end{equation}
where $p_i'$s are its distinct prime factors and $s_i'$s are their respective multiplicities. Let $\omega(n)$ denote the total number of distinct prime factors in the factorization of $n$. Thus, $\omega(n) =r$. The average distribution of $\omega(n)$ over natural numbers is well-known (see \cite[Theorem 430]{hw}):
\begin{equation}\label{ave_omega}
\sum_{n \leq x} \omega(n) = x \log \log x + B_1 x + O \left( \frac{x}{\log x}\right),
\end{equation}
where $B_1$ is the Mertens constant given by
\begin{equation}\label{B1}
B_1 = \gamma + \sum_p \left( \log \left( 1 - \frac{1}{p} \right) + \frac{1}{p} \right),
\end{equation}
with $\gamma \approx 0.57722$, the Euler-Mascheroni constant, and where the sum runs over all primes $p$. 

Let $h \geq 2$ be an integer. Let $n$ be a natural number with the factorization given in \eqref{factorization}. We say $n$ is $h$\textit{-free} if $s_i \leq h-1$ for all $i \in \{1, \cdots, r\}$, and we say $n$ is $h$\textit{-full} if $s_i \geq h$ for all $i \in \{1, \cdots, r\}$. Let $\mathcal{S}_h$ denote the set of $h$-free numbers and let $\mathcal{N}_h$ denote the set of all $h$-full numbers. 
Let $\gamma_{0,h}$ be the constant defined as
\begin{equation}\label{gamma0h}
\gamma_{0,h} := \prod_{p} \left( 1 + \frac{p-p^{1/h}}{p^2(p^{1/h} - 1)} \right),
\end{equation}
where the product runs over all primes $p$, and let $\mathcal{L}_h(r)$ be the convergent sum defined for $r > h$ as
\begin{equation}\label{lhr}
     \mathcal{L}_h(r) := \sum_p \frac{1}{p^{(r/h)-1} \left( p - p^{1-1/h} + 1 \right)}.
\end{equation}
In \cite[Theorem 1.1 and Theorem 1.2]{dkl2}, the authors proved the following distribution results for $\omega(n)$ restricted to the sets of $h$-free numbers and $h$-full numbers:
\begin{equation}\label{hfreeomega}
\sum_{\substack{n \leq x \\ n \in \mathcal{S}_h}} \omega(n) = \frac{1}{\zeta(h)} x \log \log x + 
 \left( B_1 - \sum_{p} \frac{p-1}{p(p^h -1)} \right) \frac{x}{\zeta(h)}
+ O_h \left( \frac{x}{\log x}\right),
\end{equation}
and
\begin{align}\label{hfullomega}
\sum_{\substack{n \leq x \\ n \in \mathcal{N}_h}} \omega(n) & =  \gamma_{0,h} x^{1/h} \log \log x + \left( B_1 - \log h + \mathcal{L}_h(h+1) - \mathcal{L}_h(2h) \right) \gamma_{0,h} x^{1/h} \notag \\
& \hspace{.5cm} + O_h \left( \frac{x^{1/h}}{\log x} \right),
\end{align}
where $\zeta(s)$ represents the classical Riemann $\zeta$-function, and where $O_h$ denotes big-O with the implied constant depending on $h$.
Given natural numbers $n$ and $k$, let $\omega_k(n)$ denote the number of distinct prime factors of $n$ with multiplicity $k$. Note that
$$\omega(n) = \sum_{k \geq 1} \omega_k(n).$$
Let $P(k)$ denote the convergent sum given by
\begin{equation}\label{Pk}
P(k):= \sum_p \frac{1}{p^k}.
\end{equation} 
Using this definition, the distributions of average value of $\omega_k(n)$ over natural numbers were proved by Elma and the third author \cite[Theorem 1.1]{el} as
\begin{equation}\label{omega1eq}
\sum_{n \leq x} \omega_1(n) = x \log \log x + \left( B_1 -P(2) \right) x + O \left( \frac{x}{\log x} \right),
\end{equation} 
and for $k \geq 2$, 
\begin{equation}\label{omegakeq}
 \sum_{n \leq x} \omega_k(n) = \left( P(k) - P(k+1) \right) x + O \left( x^\frac{k+1}{3k-1} (\log x)^2 \right).
\end{equation} 
The results in \eqref{ave_omega}, \eqref{omega1eq} and \eqref{omegakeq} suggest that $\omega(n)$ and $\omega_1(n)$ share a similar asymptotic distribution with only a difference in the coefficient of the second main term, whereas the average distribution of $\omega_k(n)$ with $k \geq 2$ is smaller than that of $\omega(n)$. We verify this in our work as well. In the next two theorems, we show that when restricted to the set of $h$-free numbers, $\omega_1(n)$ behaves similarly to $\omega(n)$ and $\omega_k(n)$ with $k \geq 2$ exhibits a smaller asymptotic size that $\omega(n)$. We begin by studying the distributions of $\omega_1(n)$ over $h$-free numbers. 
We define the constants
\begin{equation}\label{C1}
    C_1 :=  B_1 - \sum_{p} \frac{p^{h-1}-1}{p(p^h -1)},
\end{equation}
and
\begin{equation*}
    C_2 := C_1^2 + C_1 - \zeta(2) -  \sum_p  \left( \frac{p^{h-1} - p^{h-2}}{p^h - 1} \right)^2.
\end{equation*}
For the first and the second moment of $\omega_1(n)$ over $h$-free numbers, we prove:
\begin{thm}\label{hfreeomega1}
Let $x > 2$ be a real number. Let $h \geq 2$ be an integer. Let $\mathcal{S}_h$ be the set of $h$-free numbers. Then, we have
$$\sum_{\substack{n \leq x \\ n \in \mathcal{S}_h}} \omega_1(n) = \frac{1}{\zeta(h)} x \log \log x + 
 \frac{C_1}{\zeta(h)} x
+ O_h \left( \frac{x}{\log x}\right),$$
and
\begin{align*}
\sum_{\substack{n \leq x \\ n \in \mathcal{S}_h}} \omega_1^2(n) =
\frac{1}{\zeta(h)}  x (\log \log x)^2 + \frac{2 C_1 + 1}{\zeta(h)} x \log \log x + \frac{C_2}{\zeta(h)} x + O_h \left( \frac{x}{\log x}\right).
\end{align*}
\end{thm} 
Next, we establish the moments of $\omega_k(n)$ with $k \geq 2$ over $h$-free numbers as the following:
\begin{thm}\label{hfreeomegak}
Let $k \geq 2$ and $h \geq 2$ be any integers. Let $\mathcal{S}_h$ be the set of $h$-free numbers. For $k \leq (h-1)$, we have
$$\sum_{\substack{n \leq x \\ n \in \mathcal{S}_h}} \omega_k(n) = 
\sum_{p} \left( \frac{p^{h} - p^{h-1}}{p^k(p^h - 1)} \right)  \frac{x}{\zeta(h)} + O_{h,k} \left( \frac{x^{1/k}}{\log x}\right),$$
and
\begin{align*}
  & \sum_{\substack{n \leq x \\ n \in \mathcal{S}_h}} \omega_k^2(n) \\
  & = \left( \left( \sum_{\substack{p}} \left( \frac{p^{h} - p^{h-1}}{p^k(p^h - 1)} \right) \right)^2 -  \sum_{\substack{p}} \left( \frac{p^{h} - p^{h-1}}{p^k(p^h - 1)} \right)^2 + \sum_{p} \left( \frac{p^{h} - p^{h-1}}{p^k(p^h - 1)} \right)  \right) \frac{x}{\zeta(h)} \\
  & \hspace{.5cm} + O_{h,k} \left( \frac{x^{1/k} \log \log x}{\log x} \right),  
\end{align*}
where $O_{h,k}$ means that the implied constant depends on $h$ and $k$.
\end{thm}
\begin{rmk}
Note that if $n \in \mathcal{S}_h$, then $\omega_k(n) = 0$ for all $k \geq h$. Thus, the distribution of $\omega_k(n)$ with $k \geq h$ over the $h$-free numbers is zero.
\end{rmk}
Next, we prove the distribution of $\omega_k(n)$ over $h$-full numbers. We notice that $\omega(n)$ and $\omega_h(n)$ have similar asymptotic distributions over $h$-full numbers. Moreover, $\omega_k(n)$ with $k > h$ has a smaller asymptotic size than $\omega(n)$. We can thus infer that the smallest power of primes defining the subset of $h$-full numbers contributes to the main term for the asymptotic distribution of $\omega(n)$ over the subset. This inference also satisfies the behavior observed over $h$-free numbers. The set of $h$-free numbers includes the first power of primes, and it is observed that $\omega(n)$ and $\omega_1(n)$, which counts prime factors with multiplicity 1, satisfy similar distributions over $h$-free numbers. To prove the distribution over $h$-full numbers, we define two new constants
\begin{equation}\label{D1}
    D_1 :=  B_1 - \log h - \mathcal{L}_h(2h),
\end{equation}
and
\begin{equation}\label{D2}
    D_2 := D_1^2 + D_1 - \zeta(2) 
     - \sum_{p} \left( \frac{p^{1/h}-1}{p^{1+1/h}-p+p^{1/h}} \right)^2,
\end{equation}
where $\mathcal{L}_h(\cdot)$ is defined in \eqref{lhr}.
For the first moments of $\omega_k(n)$ over $h$-full numbers, we prove:
\begin{thm}\label{hfullomegak}
Let $k \geq 2$ and $h \geq 2$ be any integers. Let $\mathcal{N}_h$ be the set of $h$-full numbers. Then, we have
\begin{align*}
\sum_{\substack{n \leq x \\ n \in \mathcal{N}_h}} \omega_h(n) & =  \gamma_{0,h} x^{1/h} \log \log x  + D_1 \gamma_{0,h} x^{1/h}  + O_h \left( \frac{x^{1/h}}{\log x} \right),
\end{align*}
where $\gamma_{0,h}$ is defined in \eqref{gamma0h}.

Moreover, we have
$$\sum_{\substack{n \leq x \\ n \in \mathcal{N}_h}} \omega_{h+1}(n) = \left( \mathcal{L}_h(h+1) - \mathcal{L}_h(h+2) \right)  \gamma_{0,h} x^{1/h} + O_{h} \left( x^{1/(h+1)} \log \log x \right),$$
and for $k > h+1$, we have
$$\sum_{\substack{n \leq x \\ n \in \mathcal{N}_h}} \omega_k(n) = \left( \mathcal{L}_h(k) - \mathcal{L}_h(k+1) \right)  \gamma_{0,h} x^{1/h} + O_{h,k} \left( x^{1/(h+1)} \right).$$
\end{thm}
For the second moments, we obtain:
\begin{thm}\label{hfullomegak^2}
Under the assumptions as in \thmref{hfullomegak}, we have
\begin{align*}
    \sum_{\substack{n \leq x \\ n \in \mathcal{N}_h}} \omega_h^2(n) & = \gamma_{0,h} x^{1/h} (\log \log x)^2 + \left( 2 D_1 + 1 \right) \gamma_{0,h} x^{1/h} \log \log x + 
     D_2 \gamma_{0,h} x^{1/h} \\
     & \hspace{.5cm} + O_h \left( \frac{x^{1/h} \log \log x}{\log x} \right).
\end{align*}
Moreover, for $k = h+1$, we have,  
\begin{align*}
    & \sum_{\substack{n \leq x \\ n \in \mathcal{N}_h}} \omega_{h+1}^2(n) \\
    & = \Bigg( \left( \mathcal{L}_h(h+1) - \mathcal{L}_h(h+2) \right)^2 + \mathcal{L}_h(h+1) - \mathcal{L}_h(h+2) \\
    & \hspace{.5cm} - \sum_{\substack{p}} \left( \frac{p^{1/h} - 1}{p^{1+2/h} - p^{1+1/h} + p^{2/h}} \right)^2 \Bigg) \gamma_{0,h} x^{1/h} + O_{h} \left( x^{1/(h+1)} (\log \log x)^2\right),
\end{align*}
and for $k > h+1$, we have
\begin{align*}
    & \sum_{\substack{n \leq x \\ n \in \mathcal{N}_h}} \omega_{k}^2(n) \\
    & = \Bigg( \left( \mathcal{L}_h(k) - \mathcal{L}_h(k+1) \right)^2 + \mathcal{L}_h(k) - \mathcal{L}_h(k+1)  \\
    & \hspace{.5cm} - \sum_p \left( \frac{p^{1/h} - 1}{p^{(k+1)/h} - p^{k/h} + p^{(k+1-h)/h}} \right)^2 \Bigg) \gamma_{0,h} x^{1/h}  + O_{h,k} \left( x^{1/(h+1)} \right).
\end{align*}   
\end{thm}
\begin{rmk}
    Note that if $n \in \mathcal{N}_h$, then $\omega_k(n) = 0$ for all $k \leq (h-1)$. Thus, the distribution of $\omega_k(n)$ over $h$-full numbers is zero for $k \leq (h-1)$. 
\end{rmk}

We recall the definition of normal order over a subset of natural numbers as mentioned in \cite{dkl2}. Let $S \subseteq \mathbb{N}$ and $S(x)$ denote the set of natural numbers belonging to $S$ and less than or equal to $x$. Let $f, F : S \rightarrow \mathbb{R}_{\geq 0}$ be two functions such that $F$ is non-decreasing. Then, $f(n)$ is said to have normal order $F(n)$ over $S$ if for any $\epsilon > 0$, the number of $n \in S(x)$ that do not satisfy the inequality
$$(1-\epsilon) F(n) \leq f(n) \leq (1+ \epsilon) F(n)$$
is $o(S(x))$ as $x \rightarrow \infty$.
Hardy and Ramanujan \cite{hardyram} proved that $\omega(n)$ has the normal order $\log \log n$ over naturals. In fact, the authors in \cite{dkl2} showed that $\omega(n)$ has the normal order $\log \log n$ over $h$-free and over $h$-full numbers as well. Note that the set of $h$-free numbers has positive density, and both $\omega$ and $\omega_1$ behave asymptotically similar over $h$-free numbers. Thus, the proof of $\omega_1(n)$ having normal order $\log \log n$ over $h$-free numbers follows from the classical case. In particular, one can establish that for any $\epsilon > 0$, the number of $n \in \mathcal{S}_h(x)$ that do not satisfy the inequality
$$(1-\epsilon) \log \log n \leq \omega_1(n) \leq (1+ \epsilon) \log \log n$$
is $o(|\mathcal{S}_h(x)|)$ as $x \rightarrow \infty$. On the other hand, the set of $h$-full numbers has density 0, and thus the proof of normal order of $\omega_h$ does not follow from the classical result. However, since $\omega$ and $\omega_h$ behave asymptotically similar over $h$-full numbers, the proof of $\omega_h(n)$ having normal order $\log \log n$ over $h$-full numbers can be inferred in a manner analogous to the proof presented in \cite[Theorem 1.3]{dkl2} for $\omega(n)$. In particular, one can establish that for any $\epsilon > 0$, the number of $n \in \mathcal{N}_h(x)$ that do not satisfy the inequality
$$(1-\epsilon) \log \log n \leq \omega_h(n) \leq (1+ \epsilon) \log \log n$$
is $o(|\mathcal{N}_h(x)|)$ as $x \rightarrow \infty$.

In \cite{ErdosKac}, Erd\H{o}s and Kac established a pioneering result that $\omega(n)$ obeys the Gaussian distribution over naturals. In particular, they proved
$$\lim_{x \rightarrow \infty} \frac{1}{x} \bigg| \left\{ n \leq x \ : \ \frac{\omega(n) - \log \log n}{\sqrt{\log \log n}} \leq a \right\} \bigg| = \Phi(a),$$
where 
\begin{equation}\label{phi(a)}
    \Phi(a) = \frac{1}{\sqrt{2 \pi}} \int_{-\infty}^a e^{-u^2/2} \ du.
\end{equation}
Following their work, various approaches to the Erd\H{o}s-Kac theorem have been pursued, for example, see \cite{pb,lg,gs,hh1,hh2,hh3,lz}. In \cite{dkl3}, the authors showed that $\omega(n)$ satisfies the Erd\H{o}s-Kac theorem over the subsets of $h$-free and $h$-full ideals of a number field, thus in particular, over $h$-free and $h$-full numbers. We extend this result to $\omega_1(n)$ over $h$-free numbers and $\omega_h(n)$ over $h$-full numbers. We prove the following two results:


%
%
\begin{thm}\label{erdoskacforomega1}
Let $x > 2$ be any real number and $h \geq 2$ be any integer. Let $\mathcal{S}_h(x)$ denote the set of $h$-free numbers less than or equal to $x$. Then for $a \in \mathbb{R}$, we have
$$\lim_{x \rightarrow \infty} \frac{1}{|\mathcal{S}_h(x)|} \bigg| \left\{ n \in \mathcal{S}_h(x) \ : \ \frac{\omega_1(n) - \log \log n}{\sqrt{\log \log n}} \leq a \right\} \bigg| = \Phi(a),$$
where $\Phi(a)$ is defined in \eqref{phi(a)}.
\end{thm}
\begin{thm}\label{erdoskacforomegah}
Let $x > 2$ be any real number and $h \geq 2$ be any integer. Let $\mathcal{N}_h(x)$ denote the set of $h$-free numbers less than or equal to $x$. Then for $a \in \mathbb{R}$, we have
$$\lim_{x \rightarrow \infty} \frac{1}{|\mathcal{N}_h(x)|} \bigg| \left\{ n \in \mathcal{N}_h(x) \ : \ \frac{\omega_h(n) - \log \log n}{\sqrt{\log \log n}} \leq a \right\} \bigg| = \Phi(a),$$
where $\Phi(a)$ is defined in \eqref{phi(a)}.
\end{thm}
Unlike $\omega_1(n)$, we observe that $\omega_k(n)$ for $1 < k < h$ does not have a normal order. This goes in accordance with the findings of Elma and the third author \cite{el} where they proved that $\omega_k(n)$ for $k > 1$ does not have a normal order over natural numbers. In particular, we prove:
\begin{thm}\label{nonormalorderhfree}
    For any integer $h \geq 2$ and any integer $k$ satisfying $1 < k < h$, the function $\omega_k(n)$ does not have normal order $F(n)$ for any non-decreasing function $F : \mathcal{S}_h \rightarrow \mathbb{R}_{\geq 0}$.
\end{thm}
Finally, we show that $\omega_k(n)$ for $k > h$ does not have a normal order over $h$-full numbers. In particular, we prove:
\begin{thm}\label{nonormalorderhfull}
    For any integer $h \geq 2$ and any integer $k > h$, the function $\omega_k(n)$ does not have normal order $F(n)$ for any non-decreasing function $F : \mathcal{N}_h \rightarrow \mathbb{R}_{\geq 0}$.
\end{thm}

For a natural number $n$, let $\Omega(n)$ denote the number of prime factors of $n$ counted with multiplicity. In particular, for the representation of $n$ given in \eqref{factorization}, $\Omega(n) = \sum_{i=1}^r s_i$. Let $\Omega_k(n)$ be the number of prime factors of $f$ with a given multiplicity $k \geqslant 1$. Note that $\Omega_k(n) = k \cdot \omega_k(n)$ and 
$\Omega(n)=\sum_{k\geqslant 1} \Omega_k(n)$ for all $n \in \mathbb{N}$. One can deduce similar results for $\Omega_k(n)$ as our results in this paper. In particular, one can prove that $\Omega_1(n)$ has normal order $\log \log n$ over $h$-free numbers, $\Omega_h(n)$ has normal order $h \log \log n$ over $h$-full numbers, and both satisfy the Erd\H{o}s-Kac Theorem. One can also prove that the functions $\Omega_k(n)$ with $1 < k < h$ do not have normal order over $h$-free numbers and $\Omega_k(n)$ with $k > h$ do not have normal order over $h$-full numbers.
\section{Lemmata}
First, we recall the following results necessary for our study:
\begin{lma}\label{primepower}\cite[Lemma 1.2]{aler}
If $k > 1$ be any real number. Then
$$\sum_{p \geq x} \frac{1}{p^k} = \frac{1}{(k-1)x^{k-1} (\log x)} + O \left( \frac{1}{x^{k-1} (\log x)^2} \right).$$
\end{lma}
\begin{lma}\label{primepowerleqx}\cite[Lemma 2.2]{dkl2}
Let $\alpha > 0$ be any real number satisfying $0 < \alpha < 1$. Then
$$\sum_{p \leq x} \frac{1}{p^\alpha} = O_\alpha \left( \frac{x^{1-\alpha}}{\log x} \right).$$
\end{lma}
\begin{lma}\label{needhfullomegak}
 Let $h \geq 2$ be a fixed integer. Let $y > 2$ and $r > h$ be fixed real numbers. Then
 $$\sum_{p \leq y} \frac{1}{p^{r/h} \left( 1 - p^{-1/h} + p^{-1} \right)} = \mathcal{L}_h(r) + O_{h,r} \left( \frac{1}{y^{\frac{r}{h} - 1} (\log y)}\right),$$
 where $\mathcal{L}_h(r)$ is the convergent sum defined in \eqref{lhr} as
 $$\mathcal{L}_h(r) := \sum_p \frac{1}{p^{(r/h)-1} \left( p - p^{1-1/h} + 1 \right)},$$
and $O_{h,r}$ means that the implied constant depends on both $h$ and $r$.
\end{lma}
\begin{proof}
Note that
\begin{align*}
    \sum_{p \leq y} \frac{1}{p^{r/h} \left( 1 - p^{-1/h} + p^{-1} \right)} 
    & = \mathcal{L}_h(r) + O_h \left( \sum_{p > y} \frac{1}{p^{r/h}} \right).
\end{align*}
Applying \lmaref{primepower} with $k = r/h$ completes the proof.
\end{proof}
\begin{lma}\label{sumplogp}\cite[Exercise 9.4.4]{murty}
    For $x > 2$, we have
    $$\sum_{p \leq x/2} \frac{1}{p \log (x/p)} = O \left( \frac{\log \log x}{\log x} \right).$$
\end{lma}
\begin{lma}\label{saidak}\cite[Lemma 2.4]{dkl2}
  Let $p$ and $q$ denote prime numbers. For $x >2$, we have
\begin{equation*}
    \sum_{\substack{p,q \\ pq \leq x}} \frac{1}{pq} = (\log \log x)^2 + 2 B_1 \log \log x + B_1^2 - \zeta(2) + O \left( \frac{\log \log x}{\log x} \right).
\end{equation*}
\end{lma}
Next, we recall the following results regarding the density of certain sequences of $h$-free and $h$-full numbers:
\begin{lma}\label{restrict}\cite[Lemma 3]{jala2}
Let $x > 2$ be a real number. Let $h \geq 2$ be an integer. Let $\mathcal{S}_h$ be the set of $h$-free numbers. Let $q_1,\cdots,q_r$ be prime numbers. Then, we have
$$\sum_{\substack{n \leq x, n \in \mathcal{S}_h \\ (n,q_1) = \cdots = (n,q_r) =1}} 1 =  \prod_{i=1}^r \left( \frac{q_i^{h} - q_i^{h-1}}{q_i^h - 1} \right) \frac{x}{\zeta(h)} + O_h \left( 2^r x^{1/h} \right).$$
\end{lma} 
Let $C_{r,h}$ be a constant dependent on $r$ and $h$, defined as
\begin{equation*}
    C_{r,h} := \prod_{j=h, j \neq r}^{2h-1} \zeta(j/r),
\end{equation*}
and let $\phi_h(s)$ be a complex valued function defined on $\Re(s) > 1/(2h+3)$, satisfying the equation
\begin{equation*}
\prod_{p} \left( 1 - p^{-(2h+2)s} + \sum_{r = 2h+3}^{(3h^2+h-2)/2} a_{r,h} p^{-rs} \right) = \zeta^{-1}((2h+2)s) \phi_h(s),
\end{equation*}
where $a_{r,h}$ satisfying the identity
\begin{equation*}
    \left( 1 + \frac{v^h}{1-v} \right)(1-v^h) \cdots (1-v^{2h-1}) = 1 - v^{2h+2} + \sum_{2h+3}^{(3h^2+h-2)/2} a_{r,h} v^r.
\end{equation*}
\begin{lma}\label{Aqhlemma}\cite[Lemma 4.1]{dkl2}
Let $q_1,q_2,\cdots,q_r$ be distinct primes. Let
\begin{equation}\label{Aqh}
    A_{q_1,\cdots,q_r,h}(x) := \sum_{\substack{n \leq x, n \in \mathcal{N}_h \\ (n,q_1) = \cdots = (n,q_r) =1}} 1.
\end{equation}
For any $x >2$, we have
\begin{equation*}
    A_{q_1,\cdots,q_r,h}(x) = \gamma_{q_1,\cdots,q_r,0,h} x^{\frac{1}{h}} + \gamma_{q_1,\cdots,q_r,1,h} x^{\frac{1}{h+1}} + \cdots + \gamma_{q_1,\cdots,q_r,h-1,h} x^{\frac{1}{2h-1}} + O_h \left( x^{\eta_h} \right),
\end{equation*}
where $\frac{1}{2h+2} < \eta_h < \frac{1}{2h -1}$, and for $i \in \{0,1,\cdots,h-1\}$,
\begin{equation*}
\gamma_{q_1,\cdots,q_r,i,h} = C_{h+i,h} \frac{\phi_h(1/(h+i))}{\zeta((2h+2)/(h+i)) \left( \prod_{j=1}^r\left( 1 +  \frac{q_j^{-h/(h+i)}}{1 - q_j^{-1/(h+i)}} \right) \right)}.
\end{equation*}
%
\end{lma}
\begin{lma}\label{hfullkfree}\cite[Lemma 17]{jala2}
Let $q$ be a prime and $k > h$ be integers. Then, for any small $\epsilon > 0$, we have
$$\sum_{\substack{n \leq x \\ n \in \mathcal{N}_h \cap \mathcal{S}_k, (n , q) = 1}} 1 = \frac{1 - q^{-1/h}}{1 - q^{-1/h} + q^{-1} - q^{-k/h}} \eta_{h,k} x^{1/h} + O \left( x^{\frac{2h+1}{2h(h+1)} + \epsilon} \right),$$
where $\eta_{h,k}$ is the convergent product given by
\begin{equation}\label{etahk}
\eta_{h,k} = \prod_p \left( 1 - \frac{1}{p} \right) \left( \frac{1 - p^{-1/h} + p^{-1} - p^{-k/h}}{1 - p^{-1/h}} \right).    
\end{equation}
\end{lma}
\section{Distribution of \texorpdfstring{$\omega_k(n)$}{} over \texorpdfstring{$h$}{}-free numbers}
In this section, we study the distribution of the average value of $\omega_k(n)$ over $h$-free numbers. 
\subsection{For \texorpdfstring{$\omega_1(n)$}{}}
\begin{proof}[\textbf{Proof of \thmref{hfreeomega1}}]
Writing $n = p y$ with $(y,p)=1$, we obtain
\begin{equation*}\label{anomega11}
    \sum_{\substack{n \leq x \\  n \in \mathcal{S}_h}} \omega_1(n) = \sum_{\substack{n \leq x \\  n \in \mathcal{S}_h}} \sum_{\substack{p \\ p \Vert n}} 1 = \sum_{p \leq x} \sum_{\substack{n \leq x \\ n \in \mathcal{S}_h, p \Vert n}} 1 = \sum_{p \leq x} \sum_{\substack{y \leq x/p \\ y \in \mathcal{S}_h, (y,p) =1 }} 1.
\end{equation*}
Now, first using \lmaref{restrict} for a single prime $p$ to the above and then using \lmaref{primepowerleqx}, we obtain
\begin{align*}
    \sum_{\substack{n \leq x \\  n \in \mathcal{S}_h}} \omega_1(n)
 & = \sum_{p \leq x} \left( \frac{p^{h} - p^{h-1}}{p(p^h - 1)} \right) \frac{x}{\zeta(h)} + O_h \left(  \frac{x}{\log x} \right).
\end{align*}
Using \lmaref{primepower}, $\frac{p^{h} - p^{h-1}}{p(p^h - 1)} = \frac{1}{p} - \frac{p^{h-1} - 1}{p(p^h - 1)}$, and Mertens' second theorem  given by
\begin{equation}\label{sum1/p}
    \sum_{p \leq x} \frac{1}{p} = \log \log x + B_1 + O \left( \frac{1}{\log x} \right),
\end{equation}
we obtain
\begin{align}\label{needforomega^2}
    \sum_{p \leq x}  \left( \frac{p^{h} - p^{h-1}}{p(p^h - 1)} \right) 
    & = \log \log x + B_1 - \sum_{p} \frac{p^{h-1} - 1}{p(p^h - 1)} + O_h \left( \frac{1}{\log x} \right).
\end{align}
Plugging the above back into the previous equation completes the first part of the proof.
Next, note that
\begin{equation}\label{1sec_om_k_1}
    \sum_{\substack{n \leq x \\ n \in \mathcal{S}_h}} \omega_1^2(n) = \sum_{\substack{n \leq x \\ n \in \mathcal{S}_h}} \left( \sum_{\substack{p \\ p \Vert n}} 1 \right)^2 = \sum_{\substack{n \leq x \\ n \in \mathcal{S}_h}} \omega_1(n) + \sum_{\substack{n \leq x \\ n \in \mathcal{S}_h}} \sum_{\substack{p,q \\ p \Vert n, q \Vert n,  p \neq q }} 1,
\end{equation}
where $p$ and $q$ above denote primes. The first sum on the right-hand side above is the first moment estimated above and for the second sum, using \lmaref{restrict} for two primes $p$ and $q$, we obtain
\begin{equation}\label{1sec_om_k_imp}
    \sum_{\substack{n \leq x \\ n \in \mathcal{S}_h}} \sum_{\substack{p,q \\ p \Vert n, q \Vert n, p \neq q}} 1 = \sum_{\substack{p,q \\ p \neq q, pq \leq x}}  \left( \left( \frac{p^{h} - p^{h-1}}{p(p^h - 1)} \right) \left( \frac{q^{h} - q^{h-1}}{q(q^h - 1)} \right) \frac{x}{\zeta(h)} + O_h \left(\frac{x^{1/h}}{(pq)^{1/h}} \right) \right).
\end{equation}
Next, we bound the above error term using \lmaref{primepowerleqx} with $\alpha = 1/h$ and \lmaref{sumplogp} as the following
\begin{align}\label{1sec_om_k_n1}
   x^{1/h} \sum_{\substack{p,q \\ p \neq q, pq \leq x}} \frac{1}{(pq)^{1/h}} & = x^{1/h} \sum_{\substack{p \leq x/2}}\frac{1}{p^{1/h}} \sum_{\substack{q \leq x/p}} \frac{1}{q^{1/h}} \notag \\
    & \ll_{h} x \sum_{\substack{p \leq x/2}} \left(  \frac{1}{p \log(x/p)}\right) \notag \\
    & \ll_{h} \frac{x \log \log x}{\log x}.
\end{align}
Now, we estimate the main term in \eqref{1sec_om_k_imp}. First, we can divide the sum as
\begin{align}\label{1pkseparate}
    & \sum_{\substack{p,q \\ p \neq q, pq \leq x}} \left( \frac{p^{h} - p^{h-1}}{p(p^h - 1)} \right) \left( \frac{q^{h} - q^{h-1}}{q(q^h - 1)} \right) \notag \\
    & = \sum_{\substack{p,q \\ pq \leq x}} \left( \frac{p^{h} - p^{h-1}}{p(p^h - 1)} \right) \left( \frac{q^{h} - q^{h-1}}{q(q^h - 1)} \right) - \sum_{\substack{p \\ p \leq x^{1/2}}} \left( \frac{p^{h} - p^{h-1}}{p(p^h - 1)} \right)^2.
\end{align}
The second sum on the right-hand side above is estimated using \lmaref{primepower} as
\begin{align}\label{1p1separate}
    \sum_{\substack{p \\ p \leq x^{1/2}}} \left( \frac{p^{h} - p^{h-1}}{p(p^h - 1)} \right)^2 
    & = \sum_{\substack{p}} \left( \frac{p^{h} - p^{h-1}}{p(p^h - 1)} \right)^2 + O \left( \frac{1}{x^{1/2} \log x}\right).
\end{align}
Using $\frac{p^h - p^{h-1}}{p(p^h -1)} = \frac{1}{p} - \frac{p^{h-1}-1}{p(p^h - 1)}$, and the symmetry in $p$ and $q$, we have
\begin{align}\label{partitions}
    & \sum_{\substack{p,q \\ pq \leq x}} \left( \frac{p^{h} - p^{h-1}}{p(p^h - 1)} \right) \left( \frac{q^{h} - q^{h-1}}{q(q^h - 1)} \right) \notag \\
    & = \sum_{\substack{p,q \\ pq \leq x}} \frac{1}{pq} - 2 \sum_{\substack{p,q \\ pq \leq x}} \frac{1}{p} \left( \frac{q^{h-1} - 1}{q(q^h - 1)}\right) + \sum_{\substack{p,q \\ pq \leq x}} \left( \frac{p^{h-1} - 1}{p(p^h - 1)}\right) \left( \frac{q^{h-1} - 1}{q(q^h - 1)}\right).
\end{align}
We estimate the sums on the right-hand side above separately. For the first sum, we use \lmaref{saidak}. For the second sum, we use \lmaref{primepower}, and then \eqref{sum1/p} and the classical prime number theorem given as
\begin{equation}\label{PNT}
    \sum_{p \leq x} 1 = \frac{x}{\log x} + \left( \frac{x}{(\log x)^2}\right)
\end{equation} 
to obtain
\begin{align}\label{secondsum}
    & \sum_{\substack{p,q \\ pq \leq x}} \frac{1}{p} \left( \frac{q^{h-1} - 1}{q(q^h - 1)}\right) \notag\\
    & = \sum_{\substack{p \\ p \leq x/2}} \frac{1}{p} \left( \sum_{p} \left( \frac{p^{h-1} - 1}{p(p^h - 1)}\right) + O \left( \frac{1}{(x/p) \log (x/p)}\right)\right) \notag \\
    & =  \sum_{p} \left( \frac{p^{h-1} - 1}{p(p^h - 1)}\right) \left( \log \log x + B_1 \right) + O \left( \frac{1}{\log x} \right).
\end{align}
For the third sum, we use \lmaref{primepower} twice and then \lmaref{sumplogp} to obtain
\begin{align}\label{thirdsum}
    \sum_{\substack{p,q \\ pq \leq x}} \left( \frac{p^{h-1} - 1}{p(p^h - 1)}\right) \left( \frac{q^{h-1} - 1}{q(q^h - 1)}\right) 
    & = \left( \sum_{p} \left( \frac{p^{h-1} - 1}{p(p^h - 1)}\right) \right)^2 + O \left( \frac{\log \log x}{x \log x}\right) .
\end{align}
Combining \eqref{partitions}, \eqref{secondsum}, \eqref{thirdsum}, and \lmaref{saidak}, we obtain
\begin{align*}
    & \sum_{\substack{p,q \\ pq \leq x}} \left( \frac{p^{h} - p^{h-1}}{p(p^h - 1)} \right) \left( \frac{q^{h} - q^{h-1}}{q(q^h - 1)} \right) \\
    & = (\log \log x)^2 + 2 B_1 \log \log x + B_1^2 - \zeta(2) - 2 \sum_{p} \left( \frac{p^{h-1} - 1}{p(p^h - 1)}\right) \left( \log \log x + B_1 \right) \\
    & \hspace{.5cm} + \left( \sum_{p}  \frac{p^{h-1} - 1}{p(p^h - 1)} \right)^2 + O \left( \frac{\log \log x}{\log x} \right).
\end{align*}
Combining \eqref{1sec_om_k_1}, \eqref{1sec_om_k_imp}, \eqref{1sec_om_k_n1}, \eqref{1pkseparate} and \eqref{1p1separate}, and then combining it further with the above equation and the first moment of $\omega_1(n)$ over $h$-free numbers, we obtain the second moment estimate.
\end{proof}
\subsection{For \texorpdfstring{$\omega_k(n)$}{} with \texorpdfstring{$k \geq 2$}{}}
\begin{proof}[\textbf{Proof of \thmref{hfreeomegak}}]
Using $n = p^k y$ with $(p,y) = 1$, we obtain
\begin{equation}\label{omegakmain}
\sum_{\substack{n \leq x \\ n \in \mathcal{S}_h}} \omega_k(n) = \sum_{\substack{n \leq x \\ n \in \mathcal{S}_h}} \sum_{\substack{p|n \\ p^k \Vert n}} 1 = \sum_{p \leq x^{1/k}} \sum_{\substack{n \leq x \\ n \in \mathcal{S}_h, p^k \Vert n}} 1 = \sum_{p \leq x^{1/k}} \sum_{\substack{y \leq x/p^k \\ y \in \mathcal{S}_h, (p,y) = 1}} 1.
\end{equation}
Using \lmaref{restrict} for a single prime $p$, we obtain
\begin{align}\label{omegakbound1}
\sum_{p \leq x^{1/k}} \sum_{\substack{y \leq x/p^k \\ y \in \mathcal{S}_h, (p,y) = 1}} 1 & = \sum_{p \leq x^{1/k}} \left( \frac{1}{\zeta(h)} \left( \frac{p^{h} - p^{h-1}}{p^h - 1} \right) \frac{x}{p^k} + O_h \left( \frac{x^{1/h}}{p^{k/h}}\right) \right) \notag \\
& = \sum_{p} \left( \frac{p^{h} - p^{h-1}}{p^k(p^h - 1)} \right) \frac{x}{\zeta(h)} + \sum_{p > x^{1/k}} \left( \frac{p^{h} - p^{h-1}}{p^k(p^h - 1)} \right) \frac{x}{\zeta(h)} \notag \\
& \hspace{.5cm} + O_h \left( x^{1/h} \sum_{p \leq x^{1/k}} \frac{1}{p^{k/h}}\right).
\end{align}
Using \lmaref{primepower}, we estimate the second sum in the above equation as
\begin{align}\label{omegakbound2}
\sum_{p > x^{1/k}} \left( \frac{p^{h} - p^{h-1}}{p^k(p^h - 1)} \right) \frac{x}{\zeta(h)} \ll_{h,k} \frac{x^{1/k}}{\log x}.
\end{align}
%
Using \lmaref{primepowerleqx} with $\alpha = k/h$, we obtain
\begin{equation}\label{sumk/h}
    \sum_{p \leq x^{1/k}} \frac{1}{p^{k/h}} \ll_{h,k} \frac{x^{\frac{1}{k} -\frac{1}{h}}}{\log x}.
\end{equation}
%
Combining \eqref{omegakmain}, \eqref{omegakbound1}, \eqref{omegakbound2}, and \eqref{sumk/h} completes the first part of the proof.

For the second moment, note that
\begin{equation}\label{sec_om_k_1}
    \sum_{\substack{n \leq x \\ n \in \mathcal{S}_h}} \omega_k^2(n) = \sum_{\substack{n \leq x \\ n \in \mathcal{S}_h}} \left( \sum_{\substack{p \\ p^k \Vert n}} 1 \right)^2 = \sum_{\substack{n \leq x \\ n \in \mathcal{S}_h}} \omega_k(n) + \sum_{\substack{n \leq x \\ n \in \mathcal{S}_h}} \sum_{\substack{p,q \\ p^k \Vert n, q^k \Vert n,  p \neq q }} 1,
\end{equation}
where $p$ and $q$ above denote primes. The first sum on the right-hand side above is te first moment studied above, and for the second sum, using \lmaref{restrict}, we obtain
\begin{equation}\label{sec_om_k_imp}
    \sum_{\substack{n \leq x \\ n \in \mathcal{S}_h}} \sum_{\substack{p,q \\ p^k \Vert n, q^k \Vert n, p \neq q}} 1 = \sum_{\substack{p,q \\ p \neq q, pq \leq x^{1/k}}}  \left( \left( \frac{p^{h} - p^{h-1}}{p^k(p^h - 1)} \right) \left( \frac{q^{h} - q^{h-1}}{q^k(q^h - 1)} \right) \frac{x}{\zeta(h)} + O_h \left(\frac{x^{1/h}}{(pq)^{k/h}} \right) \right).
\end{equation}
Next, we bound the above error term. We employ \lmaref{primepowerleqx} with $\alpha = k/h$, and with $\alpha = 2k/h$ when $2k <h$, \eqref{sum1/p} when $2k =h$, and $\sum_p 1/p^{2k/h} = O(1)$ when $2k >h$ below to obtain 
\begin{align}\label{sec_om_k_2}
    x^{1/h} \sum_{\substack{p,q \\ p \neq q, pq \leq x^{1/k}}}  \frac{1}{(pq)^{k/h}} & = x^{1/h} \sum_{\substack{p,q \\ pq \leq x^{1/k}}}  \frac{1}{(pq)^{k/h}} - x^{1/h} \sum_{\substack{p \leq x^{1/2k}}}  \frac{1}{p^{2k/h}} \notag \\
    & \ll_{k,h} \frac{x^{\frac{1}{k}} \log \log x}{\log x}.
\end{align}
Now, we estimate the main term in \eqref{sec_om_k_imp}. First, we can divide the sum as
\begin{align}\label{pkseparate}
    & \sum_{\substack{p,q \\ p \neq q, pq \leq x^{1/k}}} \left( \frac{p^{h} - p^{h-1}}{p^k(p^h - 1)} \right) \left( \frac{q^{h} - q^{h-1}}{q^k(q^h - 1)} \right) \notag \\
    & = \sum_{\substack{p,q \\ pq \leq x^{1/k}}} \left( \frac{p^{h} - p^{h-1}}{p^k(p^h - 1)} \right) \left( \frac{q^{h} - q^{h-1}}{q^k(q^h - 1)} \right) - \sum_{\substack{p \\ p \leq x^{1/2k}}} \left( \frac{p^{h} - p^{h-1}}{p^k(p^h - 1)} \right)^2.
\end{align}
The second sum on the right-hand side above is estimated using \lmaref{primepower} as
\begin{align}\label{p1separate}
    \sum_{\substack{p \\ p \leq x^{1/2k}}} \left( \frac{p^{h} - p^{h-1}}{p^k(p^h - 1)} \right)^2 
    & = \sum_{\substack{p}} \left( \frac{p^{h} - p^{h-1}}{p^k(p^h - 1)} \right)^2 + O_{k} \left( \frac{1}{x^{1-\frac{1}{2k}} \log x}\right).
\end{align}
For the first sum, we employ \lmaref{primepower}, then \lmaref{sumplogp}, and then again \lmaref{primepower} to obtain
\begin{align*}
    & \sum_{\substack{p,q \\ pq \leq x^{1/k}}} \left( \frac{p^{h} - p^{h-1}}{p^k(p^h - 1)} \right) \left( \frac{q^{h} - q^{h-1}}{q^k(q^h - 1)} \right) \\
    & = \sum_{\substack{p \\ p \leq x^{1/k}/2}} \left( \frac{p^{h} - p^{h-1}}{p^k(p^h - 1)} \right)  \left( \sum_{\substack{p}} \left( \frac{p^{h} - p^{h-1}}{p^k(p^h - 1)} \right) \right) + O_k \left( x^{\frac{1}{k} - 1} \sum_{\substack{p \\ p \leq x^{1/k}/2}} \frac{1}{p \log(x^{1/k}/p)}\right) \\
    & = \left( \sum_{\substack{p}} \left( \frac{p^{h} - p^{h-1}}{p^k(p^h - 1)} \right) \right)^2 + O_k \left( \frac{x^{\frac{1}{k}-1}}{\log x} \right) + O_k \left( \frac{x^{\frac{1}{k} - 1} \log \log x}{\log x} \right).
\end{align*}
Combining the last three results, we obtain
\begin{align*}
    & \sum_{\substack{p,q \\ p \neq q, pq \leq x^{1/k}}} \left( \frac{p^{h} - p^{h-1}}{p^k(p^h - 1)} \right) \left( \frac{q^{h} - q^{h-1}}{q^k(q^h - 1)} \right) \frac{x}{\zeta(h)} \\
    & = \left( \sum_{\substack{p}} \left( \frac{p^{h} - p^{h-1}}{p^k(p^h - 1)} \right) \right)^2 -  \sum_{\substack{p}} \left( \frac{p^{h} - p^{h-1}}{p^k(p^h - 1)} \right)^2 + O_{h,k} \left( \frac{x^{1/k} \log \log x}{\log x} \right).
\end{align*}
Combining the above with \eqref{sec_om_k_1}, \eqref{sec_om_k_imp}, and the first moment of $\omega_k(n)$ studied in the first part of the proof, we obtain the required result.
\end{proof}

\section{Distribution of \texorpdfstring{$\omega_k(n)$}{} over \texorpdfstring{$h$}{}-full numbers}

In this section, we study the distribution of the function $\omega_k(n)$ over $h$-full numbers. The definition of $h$-full numbers enforces that the distribution of $\omega_k(n)$ over $h$-full numbers is zero for $k \leq h-1$. Thus, we only need to study the case $k \geq h$. 
\subsection{The first moment of \texorpdfstring{$\omega_k(n)$}{} over \texorpdfstring{$h$}{}-full numbers}
\begin{proof}[\textbf{Proof of \thmref{hfullomegak}}]
Note that 
\begin{equation}\label{omegakeqhfull}
\sum_{\substack{n \leq x \\ n \in \mathcal{N}_h}} \omega_k(n) = \sum_{\substack{n \leq x \\ n \in \mathcal{N}_h}} \sum_{\substack{p|n \\ p^k \Vert n}} 1 = \sum_{p \leq x^{1/k}} \sum_{\substack{n \leq x \\ n \in \mathcal{N}_h, p^k \Vert n}} 1
= \sum_{p \leq x^{1/k}}  A_{p,h}(x/p^k),
\end{equation}
where $A_{p,h}(y)$ is defined in \eqref{Aqh}. Thus, applying \lmaref{Aqhlemma} with a single prime $p$, we have
\begin{align}\label{reghfull1}
    \sum_{\substack{n \leq x \\ n \in \mathcal{N}_h}} \omega_k(n) 
    & = \gamma_{0,h} x^{1/h} \sum_{p \leq x^{1/k}} \frac{1 - p^{-1/h}}{p^{k/h} \left( 1 - p^{-1/h} + p^{-1} \right)} + O_h \left( x^{1/(h+1)} \sum_{p \leq x^{1/k}} \frac{1}{p^{k/(h+1)}} \right).
\end{align}
The above formula for $k = h$ yields
\begin{align*}
    \sum_{\substack{n \leq x \\ n \in \mathcal{N}_h}} \omega_h(n) & = \gamma_{0,h} x^{1/h} \left( \sum_{p \leq x^{1/h}} \frac{1}{p \left( 1 - p^{-1/h} + p^{-1} \right)} - \sum_{p \leq x^{1/h}} \frac{1}{p^{1+1/h}\left( 1 - p^{-1/h} + p^{-1} \right)} \right)\\
    & \hspace{.5cm} + O_h \left( x^{1/(h+1)} \sum_{p \leq x^{1/h}} \frac{1}{p^{h/(h+1)}} \right).
\end{align*}
By \cite[(42)]{dkl2}, we have
\begin{equation}\label{req_bound_p}
    \sum_{p \leq x^{1/h}} \frac{1}{p\left( 1- p^{-1/h}+p^{-1} \right)} = \log \log x + B_1 - \log h + \mathcal{L}_h(h+1) - \mathcal{L}_h(2h) 
    + O_h \left( \frac{1}{\log x} \right).
\end{equation}
Thus, using \lmaref{primepowerleqx} with $\alpha = h/(h+1)$, and \lmaref{needhfullomegak} with $y = x^{1/h}$ and $r = h+1$, we obtain
\begin{align*}
\sum_{\substack{n \leq x \\ n \in \mathcal{N}_h}} \omega_h(n) & =  \gamma_{0,h} x^{1/h} \log \log x  + \Bigg( B_1 - \log h - \mathcal{L}_h(2h) \Bigg) \gamma_{0,h} x^{1/h} + O_h \left( \frac{x^{1/h}}{\log x} \right).
\end{align*}
Now, let's consider the case $k > h$. Rewriting \eqref{reghfull1} for $k > h$ and using \lmaref{needhfullomegak}, we obtain
\begin{align*}
    \sum_{\substack{n \leq x \\ n \in \mathcal{N}_h}} \omega_k(n) 
    & = \gamma_{0,h} x^{1/h} \left( \mathcal{L}_h(k) - \mathcal{L}_h(k+1) + O_{h,k} \left( \frac{1}{x^{\frac{1}{h} - \frac{1}{k}} (\log x)}\right)\right) \\
    & \hspace{.5cm} + O_h \left( x^{1/(h+1)} \sum_{p \leq x^{1/k}} \frac{1}{p^{k/(h+1)}} \right).
\end{align*}
Note that, for $k = h+1$,
$$\sum_{p \leq x^{1/k}} \frac{1}{p^{k/(h+1)}} = O (\log \log x)$$
and for $k > h+1$,
$$\sum_{p \leq x^{1/k}} \frac{1}{p^{k/(h+1)}} = O_k (1) .$$
Combining the above results, we obtain
$$\sum_{\substack{n \leq x \\ n \in \mathcal{N}_h}} \omega_{h+1}(n) = \left( \mathcal{L}_h(h+1) - \mathcal{L}_h(h+2) \right)  \gamma_{0,h} x^{1/h} + O_{h} \left( x^{1/(h+1)} \log \log x \right),$$
and for $k > h+1$, we obtain
$$\sum_{\substack{n \leq x \\ n \in \mathcal{N}_h}} \omega_k(n) = \left( \mathcal{L}_h(k) - \mathcal{L}_h(k+1) \right)  \gamma_{0,h} x^{1/h} + O_{h,k} \left( x^{1/(h+1)} \right).$$
This completes the proof.
\end{proof}
\subsection{The second moment of \texorpdfstring{$\omega_k(n)$}{} over \texorpdfstring{$h$}{}-full numbers}
\begin{proof}[\textbf{Proof of \thmref{hfullomegak^2}}]
Note that, for $k \geq h$, we have
\begin{equation}\label{hfullpk1}
    \sum_{\substack{n \leq x \\ n \in \mathcal{N}_h}} \omega_k^2(n) = \sum_{\substack{n \leq x \\ n \in \mathcal{N}_h}} \left( \sum_{\substack{p \\ p^k \Vert n}} 1 \right)^2 = \sum_{\substack{n \leq x \\ n \in \mathcal{N}_h}} \omega_k(n) + \sum_{\substack{n \leq x \\ n \in \mathcal{N}_h}} \sum_{\substack{p,q \\ p^k \Vert n, q^k \Vert n,  p \neq q }} 1,
\end{equation}
where $p$ and $q$ above denote primes. The first sum on the right-hand side above can be estimated using \thmref{hfullomegak} and for the second sum, we first rewrite the sum, and use \lmaref{Aqhlemma} with two distinct primes $p$ and $q$ to obtain
\begin{align}\label{hfullpk2}
    & \sum_{\substack{n \leq x \\ n \in \mathcal{N}_h}} \sum_{\substack{p,q \\ p^k \Vert n, q^k \Vert n, p \neq q}} 1 \notag \\
    & = \sum_{\substack{p,q \\ p \neq q, pq \leq x^{1/k}}} \sum_{\substack{n \leq x/(pq)^k \\ n \in \mathcal{N}_h \\ (p,n) = (q,n) = 1}} 1 \notag \\
    & = \gamma_{0,h} x^{1/h} \sum_{\substack{p,q \\ pq \leq x^{1/k}}} \frac{1}{p^{k/h} \left( 1 + \frac{p^{-1}}{1 - p^{-1/h}} \right)} \frac{1}{q^{k/h}  \left( 1 + \frac{q^{-1}}{1 - q^{-1/h}} \right)} \notag \\
    & \hspace{.5cm} - \gamma_{0,h} x^{1/h} \sum_{\substack{p}} \left( \frac{1}{p^{k/h} \left( 1 + \frac{p^{-1}}{1 - p^{-1/h}} \right)} \right)^2 + O_{h,k} \left( \frac{x^{1/2k}}{\log x}\right)
     \notag \\
    & \hspace{.5cm}  + O_h \left( x^\frac{1}{h+1}   \sum_{\substack{p,q \\ p \neq q, pq \leq x^{1/h}}} \frac{1}{p^{k/(h+1)} q^{k/(h+1)}} \right).
\end{align}
For bounding the sum in the error term above, we use \lmaref{primepowerleqx} and \lmaref{sumplogp} for $k=h$, \lmaref{saidak} for $k = h+1$, and $\sum_{p,q} \frac{1}{(pq)^{k/(h+1)}} = O(1)$ for $k > h+1$. Thus, we obtain
\begin{equation}\label{hfullpk3}
    x^\frac{1}{h+1}   \sum_{\substack{p,q \\ p \neq q, pq \leq x^{1/h}}} \frac{1}{p^{k/(h+1)} q^{k/(h+1)}}  \ll_{h,k} \begin{cases}
    \frac{x^{1/h} \log \log x}{\log x} & \text{ if } k =h, \\
    x^{1/(h+1)} (\log \log x)^2 & \text{ if } k = h+1, \text{ and }\\
    x^{1/(h+1)} & \text{ if } k > h+1.
    \end{cases}
\end{equation}
We study the first sum in the main term in \eqref{hfullpk2} by dividing it into cases. We begin with $k=h$ and estimate the sum
$$\sum_{\substack{p,q \\ pq \leq x^{1/h}}} \frac{1 - p^{-1/h}}{p \left( 1 -p^{-1/h} + p^{-1} \right)} \frac{1 - q^{-1/h}}{q \left( 1 - q^{-1/h} + q^{-1} \right)}.$$
Note that
$$\frac{1 - p^{-1/h}}{p \left( 1 -p^{-1/h} + p^{-1} \right)} = \frac{1}{p} - \frac{1}{p^2 \left( 1 -p^{-1/h} + p^{-1} \right)}.$$
Using this, a similar result for a prime $q$ and the symmetry in primes $p$ and $q$, we expand the previous sum as
\begin{align*}
    & \sum_{\substack{p,q \\ pq \leq x^{1/h}}} \frac{1 - p^{-1/h}}{p \left( 1 -p^{-1/h} + p^{-1} \right)} \frac{1 - q^{-1/h}}{q \left( 1 - q^{-1/h} + q^{-1} \right)} \\
    & = \sum_{\substack{p,q \\ pq \leq x^{1/h}}} \frac{1}{pq} - 2 \sum_{\substack{p,q \\ pq \leq x^{1/h}}} \frac{1}{p q^2 \left( 1 -q^{-1/h} + q^{-1} \right)} \\
    & \hspace{.5cm} + \sum_{\substack{p,q \\ pq \leq x^{1/h}}} \frac{1}{p^2 \left( 1 -p^{-1/h} + p^{-1} \right)} \frac{1}{q^2 \left( 1 -q^{-1/h} + q^{-1} \right)}.
\end{align*}
We bound the first sum above using \lmaref{sumplogp}. For the second sum, we have
\begin{align*}
    & \sum_{\substack{p,q \\ pq \leq x^{1/h}}} \frac{1}{p q^2 \left( 1 -q^{-1/h} + q^{-1} \right)} \\
    & = \sum_{\substack{p \\ p \leq x^{1/h}/2}} \frac{1}{p} \sum_{\substack{q \\ q \leq x^{1/h}/p}} \frac{1}{q^2 \left( 1 -q^{-1/h} + q^{-1} \right)} \\
    & = \mathcal{L}_h(2h) \left(  \log \log x + B_1 - \log h  \right) + O_h \left( \frac{1}{\log x} \right).
\end{align*}
Similarly, for the third sum, we obtain
\begin{align*}
    & \sum_{\substack{p,q \\ pq \leq x^{1/h}}} \frac{1}{p^2 \left( 1 -p^{-1/h} + p^{-1} \right)} \frac{1}{q^2 \left( 1 -q^{-1/h} + q^{-1} \right)} 
    = \left( \mathcal{L}_h(2h) \right)^2 + O_{h} \left( \frac{\log \log x}{x^{1/h} \log x} \right).
\end{align*}
Combining the last three results with \lmaref{sumplogp}, we obtain
\begin{align*}
    & \sum_{\substack{p,q \\ pq \leq x^{1/h}}} \frac{1 - p^{-1/h}}{p \left( 1 -p^{-1/h} + p^{-1} \right)} \frac{1 - q^{-1/h}}{q \left( 1 - q^{-1/h} + q^{-1} \right)} \\
    & = (\log \log x - \log h)^2 + 2 B_1 (\log \log x - \log h) + B_1^2 - \zeta(2) \\
    & \hspace{.5cm} - 2 \left( \mathcal{L}_h(2h) \left(  \log \log x + B_1 - \log h  \right) \right)  + \left( \mathcal{L}_h(2h) \right)^2 + O_h\left( \frac{\log \log x}{\log x} \right).
\end{align*}
Combining the above with \eqref{hfullpk1}, \eqref{hfullpk2}, and \eqref{hfullpk3} for $k =h$ and using \thmref{hfullomegak}, we obtain the second moment for $\omega_h(n)$ over $h$-full numbers.

Now, we focus on the case $k \geq h+1$. 
Using the definition of $\mathcal{L}_h(r)$ \eqref{lhr},
we obtain
\begin{align*}
   & \sum_{\substack{p,q \\ pq \leq x^{1/k}}} \frac{1 - p^{-1/h}}{p^{k/h} \left( 1 -p^{-1/h} + p^{-1} \right)} \frac{1 - q^{-1/h}}{q^{k/h} \left( 1 - q^{-1/h} + q^{-1} \right)} \\
   & = \left( \mathcal{L}_h(k) - \mathcal{L}_h(k+1) \right)^2 + O_{h,k} \left( \frac{x^{\frac{1}{k} - \frac{1}{h}} \log \log x}{\log x}\right).
\end{align*}
Combining the above with \eqref{hfullpk1}, \eqref{hfullpk2}, and \eqref{hfullpk3} for $k \geq h+1$ and using \thmref{hfullomegak}, for $k \geq h+1$, we obtain the required second moments for $\omega_k(n)$ for $k \geq h+1$ over $h$-full numbers.
This completes the proof.
\end{proof}
\section{The Erd\H{o}s-Kac theorems}

In this section, we establish the Erd\H{o}s-Kac theorem for $\omega_1(n)$ over $h$-free numbers and for $\omega_h(n)$ over $h$-full numbers. To do this, we employ ideas from \cite[Proof of Theorem 1.3]{el}. We prove the following two theorems:
\begin{proof}[\textbf{Proof of \thmref{erdoskacforomega1}}]
For an arithmetic function $f$ and a natural number $n \geq 3$, let $r_f(n)$ be the ratio
\begin{equation}\label{rfn}
    r_f(n) := \frac{f(n) - \log \log n}{\sqrt{\log \log n}}.
\end{equation}
In this proof, we will be using $f$ to represent $\omega$ and $\omega_1$ when necessary. For $a \in \mathbb{R}$ and a subset $S$ of natural numbers, let $S(x)$ denote the set of elements of $S$ up to $x$, and 
\begin{equation}\label{dfsxa}
    D(f,S,x,a) := \frac{1}{|S(x)|} \left| \{ n \in S(x) \ : \ r_f(n) \leq a \} \right|
\end{equation}
be the density function for sufficiently large $x$. Since $\omega_1(n) \leq \omega(n)$, thus $r_{\omega_1}(n) \leq r_\omega(n)$ for all $n \geq 3$. Therefore using $S = \mathcal{S}_h$,
\begin{equation*}\label{ke1}
    D(\omega,\mathcal{S}_h,x,a) \leq D(\omega_1,\mathcal{S}_h,x,a)
\end{equation*}
for all $x \geq 3$. Thus, by the Erd\H{o}s-Kac theorem for $\omega(n)$ over $h$-free numbers (see \cite[Theorem 1.4]{dkl3}), we have
\begin{equation}\label{ke2}
    \Phi(a) \leq \liminf_{x \rightarrow \infty} D(\omega_1,\mathcal{S}_h,x,a).
\end{equation}
For any $\epsilon > 0$, we define the set
$$A(\mathcal{S}_h,x,\epsilon) := \left\{ n \in \mathcal{S}_h(x) \ : \  \frac{\omega(n) - \omega_1(n)}{\sqrt{\log \log n}} \leq \epsilon \right\}.$$
Let $A^c(\mathcal{S}_h,x,\epsilon)$ denote the complement of $A(\mathcal{S}_h,x,\epsilon)$ inside $\mathcal{S}_h(x)$. Note that
$$r_{\omega_1}(n) = r_\omega(n) + \frac{\omega(n) - \omega_1(n)}{\sqrt{\log \log n}}.$$
Thus, using the definition of $A(\mathcal{S}_h,x,\epsilon)$, we obtain
\begin{align*}
    \{ n \in \mathcal{S}_h(x) \ : \ r_{\omega_1} \leq a \} & =  \left\{ n \in \mathcal{S}_h(x) \ : \ r_{\omega}(n) + \frac{\omega_1(n) - \omega(n)}{\sqrt{\log \log n}} \leq a \right\} \\
    & = \left\{ n \in A(\mathcal{S}_h,x,\epsilon) \ : \ r_{\omega}(n) + \frac{\omega_1(n) - \omega(n)}{\sqrt{\log \log n}} \leq a \right\} \\
    & \hspace{.5cm} \cup \left\{ n \in A^c(\mathcal{S}_h,x,\epsilon)  \ : \ r_{\omega}(n) + \frac{\omega_1(n) - \omega(n)}{\sqrt{\log \log n}} \leq a \right\}  \\
    & \subseteq \{ n \in \mathcal{S}_h(x) \ : \ r_{\omega}(n) \leq a + \epsilon \} \cup A^c(\mathcal{S}_h,x,\epsilon).
\end{align*}
Then, by the definition of $D(f,\mathcal{S}_h,x,a))$, we have
\begin{equation}\label{ke3}
    D(\omega_1,\mathcal{S}_h,x,a) \leq D(\omega,\mathcal{S}_h,x,a+\epsilon) + \frac{|A^c(\mathcal{S}_h,x,\epsilon)|}{|\mathcal{S}_h(x)|}.
\end{equation}
We intend to show that the second summand on the right-hand side above goes to 0 as $x \rightarrow \infty$. By \eqref{hfreeomega} and \thmref{hfreeomega1}, we have
$$\sum_{n \in \mathcal{S}_h(x)} (\omega(n) -\omega_1(n)) \ll \frac{x}{\zeta(h)}.$$
Additionally,
\begin{align*}
    \sum_{n \in \mathcal{S}_h(x)} (\omega(n) -\omega_1(n)) & \geq \sum_{\substack{\frac{x}{\log x} \leq n \leq x \\ n \in A^c(\mathcal{S}_h,x,\epsilon)} } (\omega(n) -\omega_1(n)) \\
    & > \epsilon  \sum_{\substack{\frac{x}{\log x} \leq n \leq x \\ n \in A^c(\mathcal{S}_h,x,\epsilon)} } \sqrt{\log \log n} \\
    & \geq \epsilon \sqrt{\log \log (x/\log x)} \ | \{ n \geq  x/ \log x  \ : \ n \in A^c(\mathcal{S}_h,x,\epsilon) \} |.
\end{align*}
The above two results imply
$$| \{ n \geq  x/ \log x  \ : \ n \in A^c(\mathcal{S}_h,x,\epsilon) \} | \ll \frac{1}{\epsilon \cdot \zeta(h)} \frac{x}{\sqrt{\log \log (x/\log x)}},$$
where the right-hand side is $o(x)$. Since $|\mathcal{S}_h(x/\log x)|$ is also $o(x)$ and $|\mathcal{S}_h(x)| \gg x/\zeta(h)$, we obtain
$$\lim_{x \rightarrow \infty} \frac{|A^c(\mathcal{S}_h,x,\epsilon)|}{|\mathcal{S}_h(x)|} = 0.$$
Finally, taking the limits as $x \rightarrow \infty$ on both sides of \eqref{ke3}, and using the above with the Erd\H{o}s-Kac theorem for $\omega(n)$ over $h$-free numbers, we obtain
$$\limsup_{x \rightarrow \infty} D(\omega_1,\mathcal{S}_h,x,a) \leq \Phi(a + \epsilon).$$
Since, $\epsilon > 0$ is arbitrary, combining the above with \eqref{ke2} yields
$$\lim_{x \rightarrow \infty} D(\omega_1,\mathcal{S}_h,x,a) = \Phi(a).$$
This completes the proof.
\end{proof}



\begin{proof}[\textbf{Proof of \thmref{erdoskacforomegah}}]
Recall the definitions for $r_f(n)$ and $D(f,S,x,a)$ from \eqref{rfn} and \eqref{dfsxa} respectively. In this proof, we will be using $S = \mathcal{N}_h$ and $f$ to represent $\omega$ and $\omega_h$ when necessary.  

Since $\omega_h(n) \leq \omega(n)$, thus $r_{\omega_h}(n) \leq r_\omega(n)$ for all $n \geq 3$. Therefore using $S = \mathcal{N}_h$,
\begin{equation*}\label{ke12}
    D(\omega,\mathcal{N}_h,x,a) \leq D(\omega_h,\mathcal{N}_h,x,a)
\end{equation*}
for all $x \geq 3$. Thus, by the Erd\H{o}s-Kac theorem for $\omega(n)$ over $h$-full numbers (see \cite[Theorem 1.5]{dkl3}, we have
\begin{equation}\label{ke22}
    \Phi(a) \leq \liminf_{x \rightarrow \infty} D(\omega_h,\mathcal{N}_h,x,a).
\end{equation}
For any $\epsilon > 0$, we define the set
$$A_h(\mathcal{N}_h,x,\epsilon) := \left\{ n \in \mathcal{N}_h \ : \ \frac{\omega(n) - \omega_h(n)}{\sqrt{\log \log n}} \leq \epsilon \right\}.$$
Let $A_h^c(\mathcal{N}_h,x,\epsilon)$ denote the complement of $A_h(\mathcal{N}_h,x,\epsilon)$ inside $\mathcal{N}_h(x)$. Note that
$$r_{\omega_h}(n) = r_\omega(n) + \frac{\omega(n) - \omega_h(n)}{\sqrt{\log \log n}}.$$
Thus, we obtain
\begin{align*}
    \{ n \in \mathcal{N}_h(x) \ : \ r_{\omega_h} \leq a \} & =  \left\{ n \in \mathcal{N}_h(x) \ : \ r_{\omega}(n) + \frac{\omega_h(n) - \omega(n)}{\sqrt{\log \log n}} \leq a \right\} \\
    & = \left\{ n \in A_h(\mathcal{N}_h,x,\epsilon) \ : \ r_{\omega}(n) + \frac{\omega_h(n) - \omega(n)}{\sqrt{\log \log n}} \leq a \right\} \\
    & \hspace{.5cm} \cup \left\{ n \in A_h^c(\mathcal{N}_h,x,\epsilon)  \ : \ r_{\omega}(n) + \frac{\omega_h(n) - \omega(n)}{\sqrt{\log \log n}} \leq a \right\}  \\
    & \subseteq \{ n \in \mathcal{N}_h(x) \ : \ r_{\omega}(n) \leq a + \epsilon \} \cup A_h^c(\mathcal{N}_h,x,\epsilon).
\end{align*}
Then, by the definition of $D(f,\mathcal{N}_h,x,a)$, we have
\begin{equation}\label{ke32}
    D(\omega_h,\mathcal{N}_h,x,a) \leq D(\omega,\mathcal{N}_h,x,a+\epsilon) + \frac{|A_h^c(\mathcal{N}_h,x,\epsilon)|}{|\mathcal{N}_h(x)|}.
\end{equation}
We again intend to show that the second summand on the right-hand side above goes to 0 as $x \rightarrow \infty$. By \eqref{hfullomega} and \thmref{hfullomegak}, we have
$$\sum_{n \in \mathcal{N}_h(x)} (\omega(n) -\omega_h(n)) \ll_h x^{1/h}.$$
Additionally,
\begin{align*}
    \sum_{n \in \mathcal{N}_h(x)} (\omega(n) -\omega_h(n)) & \geq \sum_{\substack{\frac{x}{\log x} \leq n \leq x \\ n \in A_h^c(\mathcal{N}_h,x,\epsilon)} } (\omega(n) -\omega_h(n)) \\
    & > \epsilon  \sum_{\substack{\frac{x}{\log x} \leq n \leq x \\ n \in A^c(\mathcal{N}_h,x,\epsilon)} } \sqrt{\log \log n} \\
    & \geq \epsilon \sqrt{\log \log (x/\log x)} \ | \{ n \geq  x/ \log x  \ : \ n \in A_h^c(\mathcal{N}_h,x,\epsilon) \} |.
\end{align*}
The above two results imply
$$| \{ n \geq  x/ \log x  \ : \ n \in A_h^c(\mathcal{N}_h,x,\epsilon) \} | \ll_h \frac{1}{\epsilon} \frac{x^{1/h}}{\sqrt{\log \log (x/\log x)}},$$
where the right-hand side is $o(x^{1/h})$. Since the size of the set $|\mathcal{N}_h(x/\log x)|$ is also $o(x^{1/h})$ and $|\mathcal{N}_h(x)| \gg \gamma_{0,h} x^{1/h}$, we obtain
$$\lim_{x \rightarrow \infty} \frac{|A_h^c(\mathcal{N}_h,x,\epsilon)|}{|\mathcal{N}_h(x)|} = 0.$$
Finally, taking the limits as $x \rightarrow \infty$ on both sides of \eqref{ke32}, and using the above with the Erd\H{o}s-Kac theorem for $\omega(n)$ over $h$-full numbers, we obtain
$$\limsup_{x \rightarrow \infty} D(\omega_h,\mathcal{N}_h,x,a) \leq \Phi(a + \epsilon).$$
Since, $\epsilon > 0$ is arbitrary, combining the above with \eqref{ke22} yields
$$\lim_{x \rightarrow \infty} D(\omega_h,\mathcal{N}_h,x,a) = \Phi(a).$$
This completes the proof.
\end{proof}
\section{No normal orders}
In this section, we prove \thmref{nonormalorderhfree} and \thmref{nonormalorderhfull}. The former establishes that the functions $\omega_k(n)$ with $1 < k < h$ do not have normal order over $h$-free numbers, and the latter proves that $\omega_k(n)$ with $k > h$ do not have normal order over $h$-full numbers.
\begin{proof}[\textbf{Proof of \thmref{nonormalorderhfree}}]
    We first assume that $F(n)$ is not identically 0. Then, there exists $n_0 \in \mathbb{N}$ such that $F(n_0) > 0$ for all $n \geq n_0$. For $x >2$ and $1 < k < h$, let 
    $$S_{0,k}^h(x) := \{ n \in \mathcal{S}_h(x) \ : \ \omega_k(n) = 0 \}.$$
    Note that
    $$S_{0,k}^h(x) \supseteq \mathcal{S}_{k}(x).$$
    Since $|\mathcal{S}_{k}(x)| \gg x/\zeta(k)$, thus $|S_{0,k}^h(x)| \gg x/\zeta(k)$. In particular, the set of $n \in \mathcal{S}_h(x)$ for which $F(n) > 0$ and $\omega_k(n) = 0$ is not $o(x)$. For all such $n$, notice that the inequality
    \begin{equation}\label{nonormaleq}
        |\omega_k(n) - F(n) | > \frac{F(n)}{2}
    \end{equation}
    is satisfied. Thus, we deduce that $\omega_k(n)$ does not have normal order $F(n)$ when $F(n)$ is not identically zero.

    Next, we work with the case when $F(n)$ is identically 0. Let 
    $$S_{1,k}^h(x) := \{ n \in \mathcal{S}_h(x) \ : \ \omega_k(n) = 1 \}.$$
    Using \lmaref{restrict} for the single prime $p = 2$, we deduce
    \begin{align*}
        |S_{1,k}^h(x)| & \geq \sum_{\substack{n \in \mathcal{S}_h(x) \\ p^k \Vert n \text{ for exactly one prime } p \leq x^{1/k} }} 1 \\
        & \geq \sum_{\substack{n \in \mathcal{S}_k(x/2^k) \\ (n,2) = 1}} 1 \\
        & \gg \frac{2^k - 2^{k-1}}{2^k(2^k - 1)} \frac{x}{\zeta(k)}.
    \end{align*}
    Thus, the set of $n \in \mathcal{S}_h(x)$ for which $F(n) > 0$ and $\omega_k(n) = 1$ is not $o(x)$. Also, for all such $n$, inequality \eqref{nonormaleq} is satisfied. Thus, $\omega_k(n)$ does not have normal order $F(n)$ when $F(n)$ is identically zero. This completes the proof.
\end{proof}
%
\begin{proof}[\textbf{Proof of \thmref{nonormalorderhfull}}]
    We first assume that $F(n)$ is not identically 0. Thus, there exists $n_0 \in \mathbb{N}$ such that $F(n_0) > 0$ for all $n \geq n_0$. For $x >2$ and $k > h$, let 
    $$N_{0,k}^h(x) := \{ n \in \mathcal{N}_h(x) \ : \ \omega_k(n) = 0 \}.$$
    Note that
    $$N_{0,k}^h(x) \supseteq (\mathcal{N}_h \cap \mathcal{S}_k) (x).$$
    Moreover, using \lmaref{hfullkfree} for the prime $p =2$, we obtain
    \begin{align*}
        |(\mathcal{N}_h \cap \mathcal{S}_k) (x)| & \geq \sum_{\substack{n \leq x \\ n \in \mathcal{N}_h \cap \mathcal{S}_k, (n , 2) = 1}} 1 \\
        & \gg \frac{1 - 2^{-1/h}}{1 - 2^{-1/h} + 2^{-1} - 2^{-k/h}} \eta_{h,k} x^{1/h}.
    \end{align*}
    Thus, the set of $n \in \mathcal{N}_h(x)$ for which $F(n) > 0$ and $\omega_k(n) = 0$ is not $o(x^{1/h})$. Also, for all such $n$, inequality \eqref{nonormaleq} is satisfied. Thus, $\omega_k(n)$ does not have normal order $F(n)$ when $F(n)$ is not identically 0.

    Next, we work with the case when $F(n)$ is identically 0. Let 
    $$N_{1,k}^h(x) := \{ n \in \mathcal{N}_h(x) \ : \ \omega_k(n) = 1 \}.$$
    Using \lmaref{Aqhlemma} for the single prime $p = 2$, we deduce
    \begin{align*}
        |N_{1,k}^h(x)| & \geq \sum_{\substack{n \in \mathcal{N}_h(x) \\ p^k \Vert n \text{ for exactly one prime } p \leq x^{1/k} }} 1 \\
        & \geq \sum_{\substack{n \leq x/2^k \\ n \in \mathcal{N}_h \cap \mathcal{S}_k, (n,2) = 1}} 1 \\
        &  \gg \frac{1 - 2^{-1/h}}{2^{k/h}(1 - 2^{-1/h} + 2^{-1} - 2^{-k/h})} \eta_{h,k} x^{1/h},
    \end{align*}
    where $\eta_{h,k}$ is given by \eqref{etahk}. Therefore, the set of $n \in \mathcal{N}_h(x)$ for which $F(n) > 0$ and $\omega_k(n) = 1$ is not $o(x^{1/h})$. Also, for all such $n$, inequality \eqref{nonormaleq} is satisfied. Thus, $\omega_k(n)$ does not have normal order $F(n)$ when $F(n)$ is identically zero. This completes the proof.
\end{proof}

In this work, we establish that $\omega_1(n)$ has normal order $\log \log n$ and also satisfies the Erd\H{o}s-Kac theorem over $h$-free. Similarly, $\omega_h$ has normal order $\log \log n$ and also satisfies the Erd\H{o}s-Kac theorem over $h$-full numbers. We also proved that $\omega_k(n)$ with $1 < k < h$ do not have normal order over $h$-free numbers and $\omega_k(n)$ with $k > h$ do not have normal order over $h$-full numbers. These results can be generalized to a general number field. The authors have been working on this and will report their findings in a future article. Note that the function field analog of this research has been studied by G\'omez and Lal\'in \cite{lg}.

\section{Acknowledgement}

The authors would like to thank Matilde Lal\'in for the helpful discussions.

\bibliographystyle{plain} 

\end{document}